\documentclass{amsart}

\usepackage{hyperref}

\usepackage{graphicx}

\usepackage{xcolor}

\usepackage{amsmath, amssymb, amsfonts, amsthm, fullpage}
\usepackage{graphics,graphicx}
\usepackage{accents}
\usepackage{color}

\usepackage{algorithm}
\usepackage{algorithmicx}
\usepackage{algpseudocode}

\usepackage[margin=0.9in]{geometry}

\newcommand{\norma}[1]{{\left\vert\kern-0.25ex\left\vert\kern-0.25ex\left\vert #1
    \right\vert\kern-0.25ex\right\vert\kern-0.25ex\right\vert}}

\newcommand{\by}{\mathbf{y}}

\newcommand{\bb}{\mathbf{b}}

\newcommand{\bo}{\mathbf{0}}

\newcommand{\bx}{\mathbf{x}}
\newcommand{\bF}{\mathbf{F}}

\newcommand{\bu}{\mathbf{u}}

\newcommand{\bv}{\mathbf{v}}
\newcommand{\bA}{\mathbf{A}}
\newcommand{\bB}{\mathbf{B}}

\newcommand{\bD}{\mathbf{D}}

\newcommand{\bG}{\mathbf{G}}

\newcommand{\bM}{\mathbf{M}}
\newcommand{\bN}{\mathbf{N}}

\newcommand{\bU}{\mathbf{U}}
\newcommand{\bV}{\mathbf{V}}

\newcommand{\bI}{\mathbf{I}}

\newcommand{\tbn}[1]{{\left\vert\kern-0.25ex\left\vert\kern-0.25ex\left\vert #1 \right\vert\kern-0.25ex\right\vert\kern-0.25ex\right\vert}}

\newtheorem{remark}{Remark}[section]
\newtheorem{lemma}{Lemma}[section]

\newtheorem{theorem}{Theorem}[section]

\newtheorem{definition}{Definition}[section]

\AtBeginDocument{%
  \setlength{\belowdisplayskip}{5pt} \setlength{\belowdisplayshortskip}{5pt}
  \setlength{\abovedisplayskip}{5pt} \setlength{\abovedisplayshortskip}{5pt}
}

\title{On iterative methods based on Sherman-Morrison-Woodbury splitting}

\author{Dimitrios Mitsotakis}
\address{\textbf{D.~Mitsotakis:} Victoria University of Wellington, School of Mathematics and Statistics, PO Box 600, Wellington 6140, New Zealand}
\email{dimitrios.mitsotakis@vuw.ac.nz}

\date{\today}

\begin{document}

 
\begin{abstract}

We consider a new splitting based on the Sherman-Morrison-Woodbury formula, which is particularly effective with iterative methods for the numerical solution of large linear systems. These systems involve matrices that are perturbations of circulant or block circulant matrices, which commonly arise in the discretization of differential equations using finite element or finite difference methods. We prove the convergence of the new iteration without making any assumptions regarding the symmetry or diagonal-dominance of the matrix. 

To illustrate the efficacy of the new iteration we present various applications. These include extensions of the new iteration to block matrices that arise in certain saddle point problems as well as two-dimensional finite difference discretizations. The new method exhibits fast convergence in all of the test cases we used. It has minimal storage requirements, straightforward implementation and compatibility with nearly circulant matrices via the Fast Fourier Transform. For this reasons it can be a valuable tool for the solution of various finite element and finite difference discretizations of differential equations.
\end{abstract}

\maketitle

\section{Introduction}

Classical iterative methods for approximating solutions of linear systems are based on the splitting of matrices $\bA=\bM-\bN$ and converge whenever the iteration matrix $\bG=\bM^{-1}\bN$ has spectral radius $\rho(\bG)<1$. The spectral radius $\rho(\bG)$ also characterizes the speed of convergence of the iterative method, and the smaller it is, the faster the corresponding iterative method converges. Such splittings rely heavily on the invertibility of $\bM$. On the other hand, most commonly used methods for the numerical solution of differential equations such as finite difference and finite element methods lead to discretizations with large and sparse linear systems. When structured rectangular grids are used the coefficient matrix $\bA$ of such linear systems is usually a perturbation of a circulant matrix $\bM$, i.e. we can write $\bA=\bM-\bN$, with $\bN$ a sparse matrix. Such systems are usually solved with the help of direct methods for band matrices such as the banded $LU$ or Cholesky decomposition and in the case of circulant matrices with the Sherman-Morrison-Woodbury algorithm \cite{KMN1989} or with direct inversion of its diagonalization via the Fast Fourier Transform(FFT) \cite{Chen1987}. In addition to the previously mentioned classical discretizations, other $H^1$-Galerkin/mixed Finite element discretizations lead to saddle point problems with block nearly circulant matrices \cite{GH2006,MRKS2021}. General saddle point problems can be very challenging and various methods have been proposed for their numerical approximation of their solutions \cite{BGL2005,Brenner2009,Hadj2016,HM2021}. Among these methods we just mention the Uzawa-type iterative methods and other Krylov methods such as the conjugate gradient and GMRES method \cite{Brenner2009,SM1986}.

In this paper we study iterative methods based on a new splitting $\bA=\bM-\bN$. The analysis is not limited to perturbations of circulant or symmetric matrices: We consider general perturbations of a matrix $\bM$, which we call nearly $\bM$ matrices. The definition of a nearly $\bM$ matrix is based on the Sherman-Morrison-Woodbury formula, \cite{Golub12}:
\begin{definition}\label{def:nearcirc}
A matrix $\bA\in\mathbb{C}^{n,n}$ is near a matrix $\bM$ (we say nearly $\bM$ matrix) if there is an invertible matrix $\bM$ and a matrix $\bN=\bU\bV^T$ for appropriate $\bU,\bV$ such that
$\bA=\bM-\bN$, and with spectral radius $\rho(\bV^T\bM^{-1}\bU)<1$.
\end{definition}
Note that if the spectral radius $\rho(\bV^T\bM^{-1}\bU)<1$, then $\det(\bI-\bV^{T}\bM^{-1}\bU)\not= 0$ and the Sherman-Morrison-Woodbury theorem ensures that $\bA$ is invertible \cite{Golub12}.
An example of significant interest is the case where $\bM$ is circulant as we mentioned before. In such a case we will call the matrix $\bA$ {\em nearly circulant matrix}. We know that circulant matrices are invertible and its inversion can be performed as follows: Suppose that we want to solve the system $\bM\bx=\bb$ where $\bM$ is circulant, then the diagonalization of $\bM$ is $\bM=\mathcal{F}\bD\mathcal{F}^{-1}$ where $\mathcal{F}$ is the so-called Fourier matrix and $\bD$ a diagonal matrix with entries the discrete Fourier transform of the first row (or column) of $\bM$. Therefore, the solution of the linear system can be expressed as $\bx=\mathcal{F}({\bD}^{-1}\mathcal{F}^{-1}\bb)$, which has a trivial implementation via the FFT, \cite{Chen1987}.

Some of the main advantages of the proposed iterative method are the following:
\begin{itemize}
    \item There is no requirement for symmetric, sparse or diagonaly dominant matrix
    \item It can become faster using sparse matrix-vector multiplication  or parallel computing
    \item In the case of nearly circulant matrices it can be implemented seamlessly via FFT and it can be very fast.
\end{itemize}

In Section \ref{sec:method} we study the new iteration and its convergence. Applications with various finite element and finite difference discretizations are presented in Section \ref{sec:appl1}. The numerical experiments were performed using our implementations in Python programming language, and thus the times reported in this paper are indicative. Our implementations can be found in \cite{gitM}.

\section{The new splitting}\label{sec:method}

Let $\bA$ be near $\bM$. Using standard notation of \cite{Varga} for iterative methods, we consider the splitting $\bA=\bM-\bN$ where $\bN=\bU\bV^T$. For this particular splitting, we define the iterative method as usual 
\begin{equation}\label{eq:basic}
\bx^{(k+1)}=\bM^{-1}(\bN\bx^{(k)}+\bb)\quad \text{for $k=0,1,\dots$}\ ,
\end{equation}
for $\bx^{(0)}$ a given initial guess of the solution. We will call this new method Sherman-Morrison-Woodbury (SMW) iterative method because of the particular splitting we used.

Before proving the convergence of the new iterative scheme, and for the sake of completeness, we present a generalization of the so-called {\em matrix determinant lemma} of \cite{Ding2,DZ2007} for more general rank-$r$ modifications:
\begin{lemma}\label{lem:lemma1}
If $\bM\in\mathbb{C}^{n,n}$ is invertible, and $\bU,\bV\in\mathbb{C}^{n,n}$ matrices, then
\begin{equation}
    \det(\bM-\bU\bV^T)=\det(\bI-\bV^T\bM^{-1}\bU)\det(\bM)\ .
\end{equation}
\end{lemma}
\begin{proof}
The proof follows from the identity
$$
\begin{pmatrix}
\bM & \bU\\
\bV^T &\bI
\end{pmatrix} =
\begin{pmatrix}
\bI & \bo\\
\bV^T\bM^{-1} & \bI
\end{pmatrix}
\begin{pmatrix}
\bM & \bo\\
\bo & \bI-\bV^T\bM^{-1}\bU
\end{pmatrix}
\begin{pmatrix}
\bI & \bM^{-1}\bU\\
\bo  & \bI
\end{pmatrix}\ .
$$
The determinants of the first and last matrix on the right-hand side are equal to 1, while the determinant in the middle of this formula can be expressed as the desired product.
\end{proof}

Using the previous {\em matrix determinant lemma} we can prove the convergence of the new iterative method:
\begin{theorem}\label{thm:main1}
If $\bA$ is nearly $\bM$, then the SMW iteration converges.
\end{theorem}
\begin{proof}
In order to prove that the SMW iteration converges we will show that the spectral radius of the iteration matrix $\bG=\bM^{-1}\bN$ is less than 1, i.e. $\rho(\bM^{-1}\bN)<1$. If $\lambda$ is an eigenvalue of the iteration matrix $\bG$, then
$$\begin{aligned}
0=\det(\bM^{-1}\bN-\lambda \bI)& = \det(\bM^{-1}\bU\bV^T-\lambda \bI)=\det(-\lambda\bM^{-1})\det(\bM-\frac{1}{\lambda}\bU\bV^T)\\
\text{(By Lemma \ref{lem:lemma1})} &=\det(-\lambda \bM^{-1})\det(\bI-\frac{1}{\lambda}\bV^T\bM^{-1}\bU)\det(\bM)=\det(\bV^T\bM^{-1}\bU-\lambda \bI)\ .
\end{aligned}
$$
Thus, the eigenvalues of $\bM^{-1}\bN$ coincide with the eigenvalues of $\bV^T\bM^{-1}\bU$. By the Definition \ref{def:nearcirc} these eigenvalues are  $|\lambda|<1$. Thus, the new iterative method converges.
\end{proof}

\begin{remark}
One can define a matrix $\bA$ to be nearly $\bM$ as the perturbation $\bA=\bM-\bu\bv^T$, where $\bu,\bv\in\mathbb{R}^n$ with $|\bv^T\bM^{-1}\bu|<1$. In such a case the previous theorem holds true again and we can obtain the eigenvalue of the iteration matrix $\bG$ to be $\lambda=\bv^T\bM^{-1}\bu$.
\end{remark}

In the previous analysis $\bM$ does not need to be a circulant matrix but any matrix that can be inverted easily. On the other hand the inversion of a circulant matrix can be implemented trivially using the Fast Fourier Transform. Specifically, if $\bM$ is circulant, then the iterative method (\ref{eq:basic}) can be written as
\begin{equation}\label{eq:newmeth} 
\bx^{(k+1)}=\mathcal{F}\bD^{-1}\mathcal{F}^{-1}(\bN\bx^{(k)}+\bb)\ ,
\end{equation}
where
$\bM=\mathcal{F}\bD\mathcal{F}^{-1}$. Recall that the diagonal matrix $\bD$ is the FFT of the first row of $\bM$.
The computational cost of the inversion of $\bM$ is $O(n\log(n))$ while the matrix-vector multiplications on the right-hand side of (\ref{eq:newmeth}) can be performed using sparse matrix multiplication algorithms. The minimal storage requirements of circulant matrices and the almost linear complexity of the FFT for very large values of $n$, makes the iteration (\ref{eq:newmeth}) very useful in demanding situations.

We can improve the speed of the iteration (\ref{eq:newmeth}) by introducing an extrapolated parameter $\omega\in\mathbb{C}\setminus\{0\}$ and writing the extrapolated SMW (eSMW) iteration in the form
\begin{equation}\label{eq:relax}
\bx^{(k+1)}=[(1-\omega)\bI+\omega \bM^{-1}\bN]\bx^{(k)}+\omega\bM^{-1}\bb, \quad k=0,1,2\dots\ .
\end{equation}

It is known that the extrapolated iteration converges for $\max_{\lambda_i\in\sigma(\bG)}|1+\omega(\lambda_i-1)|<1$, \cite{Hadjidimos,Hadjidimos2}. The estimation of such parameter $\omega$ requires the knowledge of the eigenvalues $\lambda_i$ of $\bG=\bM^{-1}\bN$. To simplify the calculations we will search for a real optimal parameter $\omega_{\rm opt}$ that maximizes the speed of convergence, and also we assume that the eigenvalues $\lambda_i$ are real with $\lambda_{\min}$ and $\lambda_{\max}$ the minimum and maximum eigenvalues respectively. Since $\max_i|\lambda_i|<1$, the convergence of the extrapolated iteration requires that the eigenvalues $1+\omega(\lambda_i-1)$ of the iteration matrix $\bG_\omega=[(1-\omega)\bI+\omega \bG]$ are less than $1$. This implies that
$$0<\omega<\frac{2}{1-\lambda_{\min}}\ .$$
Following \cite{Saad2003} we compute the optimal value $\omega_{\text{opt}}$ to be
\begin{equation}\label{eq:omopt}
\omega_{\text{opt}}=\frac{2}{2-(\lambda_{\min}+\lambda_{\max})}\ .
\end{equation}
This value is the optimal real value $\omega$ that satisfies the minimization problem
$$\min_{\omega\in (0,2/(1-\lambda_{\min}))}\max_{\lambda\in \rho(G)}|1+\omega(\lambda-1)|\ .$$ 
Recognizing the difficulty of finding the optimal relaxation parameter, in most of the experiments we estimated the parameter using experimentation, while in some case we estimated the eigenvalues numerically, and then computed the optimal values $\omega_{\text{opt}}$.

\section{Applications}\label{sec:appl1}

Although we let the definition of a nearly $\bM$ matrices to be very generic, we focus here on linear systems that occur in finite element and finite difference discretizations of differential equations with nearly circulant matrices. In all cases we can write $\bN=\bU^T\bV$ and we can verify numerically that the assumptions on matrices $\bU$ and $\bV$ are satisfied. In some case we present an indicative comparison of the  performance of the new iteration with the classical Gauss-Seidel and GMRES methods, \cite{SM1986}. In all the experiments we considered as initial guess of the solution the zero vector. 

\subsection{Applications to finite element methods}\label{sec:appl1a}

First, we demonstrate the applicability of the previous iterations  by solving the classical boundary value problem 
$$
 -u''(x)+u(x)=f(x) \quad \text{for $x\in (0,1)$} \quad \text{with $u(0)=u(1)=0$}\ ,
$$
using finite element methods with spline elements. In particular we consider a uniform partition $0=x_0<x_1<\cdots<x_n=1$ of $[0,1]$ with stepsize $h=x_{i+1}-x_i$ and the space
$$S_h^1=\{\phi\in C([0,1])~:~ \phi(0)=\phi(1)=0 ~ \text{and} ~ \phi\in P^1([x_i,x_{i+1}])~\text{for}~ i=0,1,\dots,n-1\}\ ,$$
where $P^r$ denotes the space of polynomials of degree $r$.
Using the standard basis splines $\phi_i$ of continuous piecewise linear elements \cite{Mitsotakis2023}, the discrete problem is based on the computation of the $H^1$-projection $u_h\in S_h^1$ such that
$$\int_0^1 u_h'\phi'+u_h\phi~dx=\int_0^1f\phi\quad\text{for all}~ \phi\in S_h^1\ .$$
Expressing $u_h=\sum_i c_i\phi_i(x)$ and taking $\phi=\phi_j$, the discrete problem becomes equivalent to a banded linear system with matrix $$ \bA=\mathcal{K}+\mathcal{M}={\tiny\frac{1}{h}  \begin{pmatrix}
1 & -1  &  &  &    \\
-1 & 2 & -1 &  &   \\
 & \ddots & \ddots & \ddots &   \\
 &  & -1 & 2 & -1   \\
 &  &  &   -1 & 1 
\end{pmatrix}+
\frac{h}{6}
\begin{pmatrix}
2 & 1  &  &  &    \\
1 & 4 & 1 &  &   \\
 & \ddots & \ddots & \ddots &   \\
 &  & 1 & 4 & 1   \\
 &  &  &  1 & 2 
\end{pmatrix}}\ ,  $$
where $\mathcal{K}$ is the stiffness matrix and $\mathcal{M}$ the mass matrix.
For the sake of simplicity here we consider $h=1$ for any value of $n$, and we take appropriate right-hand side $\bb$ that leads to the exact solution $\bx=(1,1,\dots, 1)^T$. Specifically, we considered the circulant matrix defined by the vector
$$\bM=\text{circulant}(-5/6,8/3,-5/6)\ ,$$
and $\bN=\bM-\bA$, and we solve the corresponding linear system using the iterations (\ref{eq:newmeth}) and (\ref{eq:relax}).
Table \ref{tab:comp1} presents a comparison between the classical iterative method Gauss-Seidel, the GMRES and the new one. We observe that the new method requires the least iterations to converge while it is very fast. Taking advantage of parallel features of the FFT one can achieve better results.
\begin{table}[ht!]
    \centering
    \begin{tabular}{l|lcccc}
   & method & Gauss-Seidel & GMRES & SMW  (\ref{eq:newmeth}) &  eSMW($\omega=1.20$) (\ref{eq:relax})  \\ \hline
 $n=10^3$  & time in sec (iterations) & 0.462(25) & 0.017(21) & 0.007(18) & 0.006(13) \\
 $n=10^4$   & time in sec (iterations) & 3.472(25) & 0.026(20) & 0.010(18) &  0.008(13) \\
 $n=3\times 10^4$  & time in sec (iterations) & 19.43(25) & 0.092(19) & 0.038(18) & 0.029(13) \\ 
  \hline
    \end{tabular}
    \caption{Number of iterations and CPU time required for the linear finite elements matrix}
    \label{tab:comp1}
\end{table}

Note that we report only the time used for the execution of the main loop of the method. Our implementations took into account the sparsity of the matrices $\bA$, $\bM$ and $\bN$ and we used sparse matrix vector multiplications \cite{Mitsotakis2023}. It is  easy to check (numerically) that the condition of the Definition \ref{def:nearcirc} is satisfied by the matrices $\bM$ and $\bN$ that is derived using the $n\times n$ matrices $\bU$ and $\bV$ with $\bU_{1,1}=\bU_{n,n}=2$, $\bU_{1,n}=\bU_{n,1}=1$, $\bV_{1,1}=\bV_{n,n}=7/6$ $\bV_{1,n}=\bV_{n,1}=-1$, and all the other entries are zero.

In a similar manner we study a more demanding situation. We consider the finite element space of smooth cubic splines
$$S_h^3=\{\phi\in C^2[0,1]~ : ~ \phi(0)=\phi(1)=0 ~ \text{and} ~ \phi\in P^{3}([x_i,x_{i+1}]~\text{for}~ i=0,1,\dots,n-1\}\ ,$$
and we compute the corresponding mass matrix for the computation of the $L^2$-projection 
$$\int_0^1 u_h\phi~dx=\int_0^1 f\phi~dx\quad\text{for all}~ \phi\in S_h^3\ .$$
The particular mass matrix is hepta-diagonal
$${\tiny \bA=h\begin{pmatrix}
    31/140 & 773/2240 & 29/560 & 1/2240 &  &  &  &  &  & \\
    773/2240 & 41/40 & 17/32 & 3/56 & 1/2240 &  &  &  & & \\
    29/560 & 17/32 & 151/140 & 1191/2240 & 3/56 & 1/2240 & &  &  & \\
    1/2240 & 3/56 & 1191/2240 & 151/140 & 1191/2240 & 3/56 & 1/2240 &  & & \\
    & \ddots & \ddots & \ddots & \ddots & \ddots & \ddots & \ddots & \\
  &  & &   1/2240 & 3/56 & 1191/2240 & 151/140 & 1191/2240 & 3/56 & 1/2240  \\
    & &  &  &  1/2240 & 3/56  & 1191/2240 & 151/140 & 17/32 & 29/560 \\
   &  &  &  & &  1/2240 & 3/56 & 17/32 & 41/40 & 773/2240 \\
 &  &  &  &  &  &   1/2240  & 29/560 & 773/2240 & 31/140  \\
\end{pmatrix}}\ ,$$
and is not a diagonally dominant matrix. Again, for simplicity we take $h=1$ for all values of $n$, and we choose the right hand side appropriately so as the exact solution is again the vector $\bx=(1,1,\dots,1)^T$. For the new method we consider the circulant matrix $\bM$ which is defined as 
$$\bM=\text{circulant}(
1/2240, 3/56, 1191/2240, 151/140, 1191/2240, 3/56, 1/2240)\ ,
$$ and $\bN=\bM-\bA$. The results of the solution of this linear system using the new SMW splittings are summarized in Table \ref{tab:comp2}.  In this case, the Jacobi method did not converge at all. Luckily, the Gauss-Seidel method converged to the exact solution. In this experiment the new iteration had pretty much the same performance as our GMRES implementation. 

\begin{table}[ht!]
    \centering
    \begin{tabular}{l|lcccc}
   & method & Gauss-Seidel & GMRES & SMW  (\ref{eq:newmeth}) &  eSMW($\omega=1.86$) (\ref{eq:relax})  \\ \hline
  $n=10^3$ & time in sec (iterations) & 0.846(72) & 0.024(47) & 0.091(251) & 0.061(137) \\ 
 $n=10^4$ & time in sec (iterations) &  10.40(72) & 0.053(44) & 0.136(251) & 0.080(137)\\
 $n=3\times 10^4$ & time in sec (iterations) & 56.64(72) & 0.282(43) & 0.425(251) & 0.280(137)  \\ 
  \hline
    \end{tabular}
    \caption{Number of iterations and CPU time required for the cubic spline finite elements matrix}
    \label{tab:comp2}
\end{table}

As we mentioned before, the new methodology can be applied to more general matrices that are not necessarily symmetric or diagonally dominant. To demonstrate this we consider the matrix
$$\bA =  \begin{pmatrix}
5 & 3 & 2 & 2\\
1 & 4 & 3 & 2\\
2 & 1 & 4 & 3\\
4 & 2 & 1 & 5
\end{pmatrix}\ ,$$
which is nearly circulant with $\bN$ the matrix with $-1$ on its four corners and zeros elsewhere. The SMW iteration (\ref{eq:newmeth}) converges in 20 iterations. The extrapolated SMW iteration (\ref{eq:relax}) with $\omega_\text{opt}\approx 0.851$ converges in only 9 iterations, while the Jacobi method does not converge, the Gauss-Seidel method luckily requires 78 iterations and the GMRES required only 3 iterations. The parameter $\omega_\text{opt}$ was computed using the formula (\ref{eq:omopt}).

\begin{remark}
Note that the cases we present in this paper can be handled by other numerical methods and are commonly encountered in finite element and finite difference implementations. However, there are cases where the new method outperforms other methods. For example, consider a dense $n\times n$  circulant matrix $\bM$ (for large $n$), and as $\bN$ any sparse matrix that satisfies the requirement stated in Theorem \ref{thm:main1}. We tested such a case with $\bM$ constructed by a random vector with entries following the uniform distribution over $[1,2)$. In such a case, the method presented here is the only trivial approach to handle it even for $n=10^6$. This particular example can be found in \cite{gitM} where the resulting matrix $\bA$ is ill-conditioned but still the new method converges almost instantaneously within a few iterations. On the other hand, all the other numerical methods we tested (including the GMRES implementation of the Python module \texttt{scipy.sparse.linalg}) failed to converge even when $n=100$. The particular experiment can be found in \cite{gitM}.
\end{remark}

\begin{remark}
We underline that the performance of all the methods we used can be improved using various techniques (parallelization,  preconditioning, etc.) and even with better implementations. For this reason, these experiments serve as indication of the performance of the new iteration. On the other hand, the simplicity of the new method and its overall good performance.
\end{remark}

\subsection{Extensions to saddle point problems}\label{sec:appl1b}

In this section we consider a linear system that can be derived when we use a mixed formulation of a  finite element method for some nonlinear and dispersive wave equations with zero Dirichlet boundary conditions \cite{MRKS2021}. Consider the linear system 
$$
\begin{pmatrix}
\bA_1 & \bB_1\\
\bB_2 & \bA_2
\end{pmatrix}
\begin{pmatrix} 
\bx\\
\by
\end{pmatrix}
= \begin{pmatrix}
\bb_1\\
\bb_2
\end{pmatrix}\ ,
$$
where $\bA_1$, $\bA_2$ are square $n\times n$ and $m\times m$, and $\bB_1$ and $\bB_2$ are $n\times m$ and $m\times n$ matrices respectively. The unknown vectors $\bx$ and $\by$ are accordingly $n$ and $m$-dimensional vectors. In what follows we assume for simplicity that $A_i$ and $B_i$ are invertible $n\times n$ matrices.

Taking into account the special structure of the matrices $\bA_1$ and $\bA_2$, which are assumed to be near  $\bM_1$ and $\bM_2$, respectively, we propose the following iterative method:

Write $\bA_1 = \bM_1-\bN_1$ and $\bA_2=\bM_2-\bN_2$, then we define the block Jacobi-SMW iteration
\begin{equation}\label{eq:bjac}
\begin{pmatrix}
\bM_1 & \\
 & \bM_2
\end{pmatrix}
\begin{pmatrix} 
\bx^{(k+1)}\\
\by^{(k+1)}
\end{pmatrix}=
\begin{pmatrix}
\bN_1 & -\bB_1 \\
 -\bB_2 & \bN_2 
\end{pmatrix}
\begin{pmatrix} 
\bx^{(k)}\\
\by^{(k)}
\end{pmatrix}+\begin{pmatrix}
\bb_1\\
\bb_2
\end{pmatrix}\quad \text{for $k=0,1,\dots$}\ ,
\end{equation}
and $(\bx^{(0)},\by^{(0)})^T$ is a given initial guess of the solution. If $\bM_1=\mathcal{F}_1\bD_1\mathcal{F}_1^{-1}$ and $\bM_2=\mathcal{F}_2\bD_2\mathcal{F}_2^{-1}$ are circulant matrices with $\mathcal{F}_i$ the corresponding Fourier matrices, then the inversion of the block-diagonal matrix 
$$\begin{pmatrix}
\bM_1 & \\
& \bM_2
\end{pmatrix}
=
\begin{pmatrix}
\mathcal{F}_1 & \\
&\mathcal{F}_2
\end{pmatrix}
\begin{pmatrix}
\bD_1 & \\
& \bD_2
\end{pmatrix}
\begin{pmatrix}
\mathcal{F}_1^{-1} & \\
& \mathcal{F}_2^{-1}
\end{pmatrix}\ ,$$
is trivial using FFT and in particular
the diagonal of $\bD_i$ is actually the FFT of the vector comprised the first row of $\bM_i$. Denoting the vectors $\bu^{(k)}=(\bx^{(k)},\by^{(k)})^T$ and $\bb=(\bb_1,\bb_2)^T$, and the matrices
$$\bD=\begin{pmatrix}
  \bD_1 & \\
   & \bD_2
\end{pmatrix}, \quad \bF=\begin{pmatrix}
    \mathcal{F}_1 & \\
    & \mathcal{F}_2
\end{pmatrix}, \quad \bF^{-1}=\begin{pmatrix}
    \mathcal{F}_1^{-1} & \\
    & \mathcal{F}_2^{-1}
\end{pmatrix}\quad \text{and} \quad \bN=\begin{pmatrix}
\bN_1 & -\bB_1 \\
 -\bB_2 & \bN_2 
\end{pmatrix}\ ,$$
we express the iterative method in the form
\begin{equation}
\bu^{(k+1)}=\bF \bD^{-1} \bF^{-1}(\bN\bu^{(k)}+\bb)\ .
\end{equation}
where $\bD$ is actually the FFT of the vector comprised the first row of $\bM_1$ followed by the first row of $\bM_2$.

Another extension is the block-Gauss-Seidel-SMW variant of the new iterative method, which can be formulated as
\begin{equation}\label{eq:bgs}
\begin{pmatrix}
\bM_1 & \\
 \bB_2 & \bM_2
\end{pmatrix}
\begin{pmatrix} 
\bx^{(k+1)}\\
\by^{(k+1)}
\end{pmatrix}=
\begin{pmatrix}
\bN_1 & -\bB_1 \\
  & \bN_2 
\end{pmatrix}
\begin{pmatrix} 
\bx^{(k)}\\
\by^{(k)}
\end{pmatrix}+\begin{pmatrix}
\bb_1\\
\bb_2
\end{pmatrix}\quad \text{for $k=0,1,\dots$}\ ,
\end{equation}
and with given $(\bx^{(0)},\by^{(0)})$.
This is expected to converge in less iterations than the Jacobi method since it utilises the updates $\bx^{(k+1)}$ in the computation of the $\by^{(k+1)}$. In the special case where the matrices $\bM_1$ and $\bM_2$ are nearly circulant, then the implementation is straightforward again. If the matrices $\bB_i$ are also circulant, then the method can be accelerated again using the FFT of $\bx^{(k)}$ in all the intermediate stages. In our implementation we assume that $\bB_1$ and $\bB_2$ are sparse and we use sparse matrix-vector multiplication as it is implemented in Python in CSR storage format, cf. e.g. \cite{Mitsotakis2023}. 

As an example we consider the case where $\bA_1=\bA_2=\bA$ is the mass matrix $\mathcal{M}$  with entries $\bA_{ij}=\int_0^1 \phi_i\phi_j~dx$ and $\bB_{ij}=\int_0^1 \phi_i\phi_j'~dx$, where $\phi_i$ are the standard hat basis functions of the space of continuous and piecewise-linear functions $S_h^1$. The matrix $\bB$ is tri-diagonal $\bB=\text{tridiag}(-1/2,0,1/2)$ but not circulant. Such matrices occur in mixed formulations of finite element methods like those in \cite{MRKS2021} and \cite{GH2006}. In this particular case the method converges very fast. The performance of the new method in comparison with the sparse Gauss-Seidel method and GMRES is presented in Table \ref{tab:comp3}. 

\begin{table}[ht!]
    \centering
    \begin{tabular}{l|lcccc}
   & method & Gauss-Seidel & GMRES & SMW (\ref{eq:bjac}) & SMW (\ref{eq:bgs})  \\ \hline
  $n=10^3$ & time in sec (iterations) & 0.527(25) & 0.021(22) & 0.040(39) & 0.025(37) \\ 
 $n=10^4$ & time in sec (iterations) & 10.11(25) & 0.035(21) & 0.048(39) & 0.042(37) \\
 $n=3\times 10^4$ & time in sec (iterations) & 64.62(25) & 0.266(21) & 0.152(39) & 0.153(37)   \\ 
  \hline
    \end{tabular}
    \caption{Number of iterations and CPU time for linear elements in a mixed formulation}
    \label{tab:comp3}
\end{table}

The methodology presented in this section can be extended to larger block matrices obtained in saddle point problems and also to two-dimensional problems with tensor products of splines in a very similar manner. A particular, classical example is presented in the next section. 

\subsection{An application to two-dimensional Poisson equation}\label{sec:appl1c}

We consider the classical five-point finite difference method for the numerical solution of the Poisson equation, \cite{Smith}, 
\begin{equation}
\begin{aligned}
& \Delta u = f, \quad (x,y)\in \Omega=[0,1]\times[0,1],\\
& u(x,y)=g,\quad (x,y)\in\partial \Omega.
\end{aligned}
\end{equation}
For simplicity we consider the uniform grid with $\Delta x=\Delta y=1/m$.  Applying the five-point stencil finite difference scheme we obtsain to a block-tridiagonal system of equations $\bA\bx=\bb$ where
$\bA$ is the $m^2\times m^2$ matrix $$
\bA=\begin{pmatrix}
\bD & \bI &   &   &\\
\bI & \bD & \bI &   & \\
  & \ddots & \ddots & \ddots & \\
  &   & \bI & \bD & \bI\\
  &   &   & \bI & \bD
\end{pmatrix}\ ,
$$
$\bD=\text{tridiag}(1,-4,1)$ is $m\times m$ matrix, and $\bI$ is the $m \times m$ identity matrix. In order to devise the block-Gauss-Seidel-SMW iteration for this particular case we write $\bD=\bM-\bN$ where $\bM$ is the circulant matrix $\bM = \text{circulant}(1,-4,1)$ and $\bN$ is such that $\bN_{i,j}=0$ for all $i,j$ except for $\bN_{1,m}=\bN_{m,1}=1$. Moreover, we write the $k$-th iteration and the right hand side as a block-column vectors:
$$\bx^{(k)}=(\bx_1^{(k)},\bx_2^{(k)}\dots,\bx_m^{(k)})^T\quad \text{and}\quad \bb=(\bb_1,\bb_2,\dots,\bb_m)^T\ .$$ 
Then the algorithm is as follows:

\begin{algorithm}
\begin{algorithmic}
\State Set the initial guess $\bx_i^{(0)}$ for $i=1,2,\dots, m$
\State Set the iteration $k=0$, the tolerance $TOL$ and the maximum number of iterations $MAXIT$
\State Initialize $Error=TOL+1$
\While {$Error>TOL$ and $k<MAXIT$}
\State Solve $\bM\bx_1^{(k+1)}=\bN \bx_1^{(k)}-\bx_2^{(k)}+\bb_1$
\For {$i=2,3,\dots, m-1$}
\State Solve $\bM\bx_i^{(k+1)}=\bN\bx_i^{(k)}-\bx_{i-1}^{(k+1)}-\bx_{i+1}^{(k)}+\bb_i$
\EndFor
\State Solve $\bM x_m^{(k+1)}=\bN \bx_m^{(k)}-\bx_{m-1}^{(k+1)}+\bb_m $
\State Set $Error=\|\bx^{(k)}-\bx^{(k+1)}\|$ and $k=k+1$
\EndWhile
\State Return the approximation $\bx^{(k)}$
\end{algorithmic}\caption{Block-Gauss-Seidel-SMW iteration for the Poisson equation}\label{alg:bisection}
\end{algorithm}

To test this algorithm we took $m=200$ (i.e. $\bA\in\mathbb{R}^{20000\times 20000}$). For simplicity, we took again $\bb$ to be the vector that results in the solution $\bx$ with $x_i=1$ for all $i$. The method with $TOL=10^{-7}$ converged within 503 iterations and it required $2.98$ seconds approximately. Here it is pointless to provide any comparison with the other classical iterative methods but we report that the python implementation of GMRES converged in 5211 iterations and it required more than 30 seconds and the actual difference between the exact solution and the numerical was of the order $10^{-5}$ even if the residual was of the order $10^{-9}$. It is obvious by the requirements of this experiment that the new method can be used in large-scale problems due to its simplicity, minimum storage requirements as well as its performance. 

The implementation of this algorithm along with implementations of the previous variants of the SMW iteration and the setup of all the examples presented here can be found in \cite{gitM}. All the experiments performed on an Apple computer with processor M1 of 2020 and 16GB RAM.

\section{Conclusions}

A new splitting based on the Sherman-Morrison-Woodbury formula was presented for the iterative solution of linear systems of certain form. After introducing the notion of nearly $\bM$ matrices, we proved that for such matrices the new method converges. The new method can be used with finite element and finite difference formulations to solve banded, non-symmetric, block systems (saddle point problems) seamlessly. The applications should not be limited to those we presented here. Among the several advantages of the new method, one is its simple implementation in case of nearly circulant matrices. The speed of the new method can be increased using classical extrapolation techniques and can compete standard iterative methods. The performance of the new iteration is increasing with the dimension of the matrix and becomes faster even than the classical GMRES for large matrices.

\section*{Acknowledgments}

The author would like to express his gratitude to Prof. A. Hadjidimos for the fruitful discussions on the subject.

\bibliographystyle{plain}

\end{document}